\documentclass{scrartcl}
\usepackage[latin9]{inputenc}
\usepackage{a4}
\usepackage{amsfonts}
\usepackage{amssymb}
\usepackage{amsmath}
\usepackage{amsthm}
\usepackage{makeidx}
\usepackage{pdflscape}
\usepackage{setspace}
\usepackage{multirow}
\usepackage{tabularx}
\usepackage{float}
\makeindex
\usepackage{nomencl}
\usepackage{geometry}
\usepackage{extarrows}
\usepackage{changepage}
\usepackage{chngcntr} 
\usepackage[pdftex]{graphicx}
\usepackage{tikz}
\usepackage[toc,page]{appendix}
\providecommand{\keywords}[1]{\textbf{\textit{Keywords: }} #1}
\theoremstyle{plain}
\newtheorem{satz}{Theorem}\numberwithin{satz}{section} %[section]
\newtheorem{lemma}[satz]{Lemma}\numberwithin{satz}{section} 
\numberwithin{satz}{section} 
\newtheorem{kor}[satz]{Corollary}\numberwithin{satz}{section} 
\theoremstyle{definition}
\newtheorem{defi}{Definition}\numberwithin{defi}{section} 
\theoremstyle{remark}
\newtheorem{conj}{Conjecture}\numberwithin{conj}{section}
\newtheorem{remark}{Remark}\numberwithin{remark}{section}
\newcommand{\cc}{{\mathbb{C}}}   % complexe Zahlen
\newcommand{\nn}{{\mathbb{N}}}   % natuerliche Zahlen
\begin{document}
\title{A note on the product of two permutations of prescribed orders}
\author{J.\ K\"onig\thanks{Universit\"at W\"urzburg, Emil-Fischer-Str.\ 30, 97074 W\"urzburg, Germany. email: joachim.koenig@mathematik.uni-wuerzburg.de}}
\maketitle
\begin{abstract}{We prove a conjecture by Stefan Kohl on the existence of triples of permutations of bounded degree with prescribed orders and product $1$.\\
More precisely, let $a,b,c$ be integers, all $\ge 2$. Then there exist elements $x,y,z\in S_{c+2}$ of orders $a$, $b$ and $c$ respectively, with $xyz=1$.\\
This result leads to an existence result for covers of the complex projective line with bounded degree and prescribed ramification indices.}
\end{abstract}
\keywords{Permutations; covering maps; Hurwitz existence problem}
\section{Tuples of permutations of prescribed order with product 1}
In \cite[Problem 18.49]{Kourovka}, S.\ Kohl conjectured the following:
\begin{conj}
\label{kohl}
Given $n\in \nn$ and $1< a,b,c\le n-2$, then there exist elements $x,y$ of $S_n$ such that $x$ has order $a$, $y$ has order $b$ and $xy$ has order $c$.\end{conj}
The conjecture has previously been verified by computer calculation for $n\le 50$, cf.\ \cite{Kohl}. 

In general, for given subsets (e.g., conjugacy classes) $C_1,...,C_r$ of the symmetric group $S_n$, it is not 
at all an easy problem to decide whether there are permutations $\alpha_i\in C_i$ with $\alpha_1\cdots \alpha_r = 1$. 
At the same time, such questions are of major interest also outside of the purely combinatorial context (see Section \ref{covers}).
\subsection{The main theorem}
In this paper, we prove the following stronger version of Conjecture \ref{kohl}:
\begin{satz}
\label{main}
Let $1 < a\le b\le c$, then there exist elements $x,y,z$ of $S_{c+2}$ with orders $|x|=a$, $|y|=b$, $|z|=c$, fulfilling $xyz=1$. Additionally, $x$, $y$ and $z$ can be required to fulfill the following:
\begin{itemize}
 \item $z$ is a $c$-cycle or the disjoint product of a $c$-cycle and a transposition.
 \item $x$ only has cycles of length $a$ and (possibly) $1$, with the exception of at most one transposition.
 \item $y$ only has cycles of length $b$ and (possibly) $1$, with the exception of at most one transposition.
 \item The ``exceptional" transposition occurs in at most one of $x$, $y$ and $z$.
 \end{itemize}
\end{satz}
First note that the requirement $a\le b\le c$ is of course not a real restriction in comparison with the original conjecture, since the product-one condition for $x$, $y$ and $z:=(xy)^{-1}$ is invariant under cyclic 
permutations of $x,y,z$; and if $(x,y,z)$ has product one, then so has $(y,x,z^x)$ - with $z$ and $z^x$ of the same cycle type.\\ 
Also, the permutation degree $c+2$ is in general best possible, as the tuple $(2^k-1,2^k-1,2^k)$ (with $k\ge 2$) shows.\\%
The extra transposition in the statement of the theorem can not be avoided in general; it is needed as a parity-check bit for the sign of $xyz$.
\subsection{Auxiliary results}%Besser!
The proof of Theorem \ref{main} will make frequent use of the {\it index of a permutation}. We therefore recall its definition:
\begin{defi}
For $\alpha\in S_n$, the index $ind(\alpha)$ is defined as $n$ minus the number of disjoint cycles of $\alpha$.\\
If $C$ is the conjugacy class of $S_n$ containing $\alpha$, define $ind(C):=ind(\alpha)$.
\end{defi}
Note that, equivalently, the index of $\alpha$ is the smallest number $k\in \nn$ such that $\alpha$ can be written as a product of $k$ transpositions.
In particular, an $m$-cycle has index $m-1$.

The following lemma is completely elementary. We state it as we will use it later in the proof of Theorem \ref{main} without further commentary.
\begin{lemma}
\label{estim}
Let $2\le k\le n\in \nn$ and let $\alpha\in S_n$ consist of $\lfloor n/k\rfloor$ $k$-cycles in disjoint cycle notation. Then it holds that
\begin{itemize}
 \item[a)] $$ind(\alpha) \ge \frac{n-1}{2},$$
with equality if and only if $n$ is odd and either $k=2$ or $k=\frac{n+1}{2}$.
\item[b)] If $n$ is even, then $$ind(\alpha) \ge \frac{n}{2},$$ with equality if and only if $k\in \{2,\frac{n}{2}+1\}$ or $(n,k)=(8,3)$.%Stimmt das so?!
\end{itemize}
\end{lemma}
\begin{proof}
For a), we can assume without loss that $k\le \frac{n+1}{2}$, as otherwise $\alpha$ is a single $k$-cycle, with index $k-1>\frac{n-1}{2}$.
It holds that $$ind(\alpha)=\lfloor n/k\rfloor \cdot (k-1) \ge \frac{n+1-k}{k}\cdot (k-1) = n+2 -(k+\frac{n+1}{k}).$$ 
The last bracket is $\le 2+\frac{n+1}{2}$, with equality at the extreme cases $k=2$ and $k=\frac{n+1}{2}$. This shows a).

For b), if $k\notin \{2,\frac{n}{2}+1\}$, we can even assume $3\le k\le \frac{n}{3}$, since $\frac{n}{3}<k\le \frac{n}{2}$ would yield $ind(\alpha)>2\cdot (\frac{n}{3}-1)$, which is $>\frac{n}{2}$ for $n>4$. But then, we have
$$ind(\alpha)\ge  n+2 -(k+\frac{n+1}{k}) \ge n+2 -(3+\frac{n+1}{3})=\frac{2}{3}(n+1)-2.$$ This is $\le \frac{n}{2}$ only for $n\le 8$, and one checks directly that $(n,k)=(8,3)$ is the only additional case that 
reaches equality.
\end{proof}

The main ingredients for the proof of Theorem \ref{main} are the strong existence results in Theorem \ref{eks} below; they were obtained by Edmonds, Kulkarni and Stong in \cite[Cor.\ 4.4.\ and Lemma 4.5]{EKS}.
As noted there, they also follow from \cite[Thm.\ 4.3]{Boc}. 
%
%Hier nochmal Text?!? Oder Satz benennen
\begin{satz}[Edmonds, Kulkarni, Stong]
\label{eks}
Let $C_1, C_2$ be conjugacy classes of $S_n$.
\begin{itemize}
 \item[a)] 
Assume that $ind(C_1)+ind(C_2)$ is of the form $n-1 + 2k$, with $k\in \nn_0$.\\
Then there exist $\alpha\in C_1$ and $\beta \in C_2$ such that $\alpha\beta$ is an $n$-cycle.
\item[b)]
Assume that $ind(C_1)+ind(C_2)$ is of the form $n + 2k$, with $k\in \nn_0$.\\
If in addition, $C_1$ and $C_2$ are not both the class of fixed point free involutions, then 
there exist $\alpha\in C_1$ and $\beta \in C_2$ such that $\alpha\beta$ is an $(n-1)$-cycle and the subgroup $\langle\alpha, \beta\rangle$ acts transitively.
\end{itemize}
%\marginpar{Gleiches scheint auch f\"ur $(n-2).2$ zu gelten!!}
\end{satz}
Note that the conditions on the sum of indices in Theorem \ref{eks} are also necessary, as can be easily deduced from the Riemann-Hurwitz formula (stated below in Theorem \ref{riem-hurw}).
\subsection{Proof of Theorem \ref{main}}
\label{proof}
Now we proceed to the proof of Theorem \ref{main}. 
Throughout, we will assume $2\le a\le b\le c\in \nn$, and we will use the following terminology: 
Let $A$ be the class of elements of $S_{c}$ with $\lfloor c/a \rfloor$ $a$-cycles and fixed points otherwise; similarly define the class $B$ via $b$ instead of $a$.
Let $C$ be the class of $c$-cycles.

We will opt to prove Theorem \ref{main} by using the classes $A$, $B$ and $C$ or by slightly adapting them.
We will break the proof into several cases, depending on the exact permutation degree that we work with (and also on the way Theorem \ref{eks} will be used). We will call the class triple $(A,B,C)$ even, if $ind(A)+ind(B)+ind(C)$ is even,
and odd otherwise. In particular, an odd class triple cannot contain elements with product $1$, so in this case, we need to modify one of the classes by removing a cycle or adding an extra transposition.%?
%
%Multiplication of permutations will always be done from left to right.%nötig?
%
\begin{lemma}%Wie formulieren? Zuerst $b<c$?
Assume that either $(A,B,C)$ is even or that $a$, $b$ and $c$ are all even and not all the same.
Then the statement of Theorem \ref{main} holds even with elements $x,y,z$ all in $S_c$.  
\end{lemma}
\begin{proof}
By Lemma \ref{estim}a), $$ind(A)+ind(B) \ge c-1,$$%mit Nummer?
and since $ind(C)=c-1$, we only need $(A,B,C)$ to be even in order to be able to apply Theorem \ref{eks}a). This shows the first case.

So let $(A,B,C)$ be odd and $a$, $b$, $c$ all be even, but not all the same; in particular, $a\le c-2$.\\
If $a$ does not divide $c$, then the elements of class $A$ have at least two fixed points, and we can therefore replace $A$ by the class $A'$ in $S_c$ containing one extra transposition compared to $A$.
The elements of $A'$ still have order $a$, and we can apply Theorem \ref{eks}a) with the class triple $(A',B,C)$.\\
If, on the other hand, $a$ properly divides $c$, let $A'$ be the class in $S_c$ with exactly one $a$-cycle less than $A$. As $a$ is even, the triple $(A',B,C)$ is again even. Furthermore, 
$$ind(A')=\underbrace{\frac{c-a}{a}}_{\text{ number of } a-\text{cycles in} A'} \cdot (a-1) = c+1-\underbrace{(a+c/a)}_{\le 2+c/2} \ge c/2-1,$$
which, together with $ind(B)\ge c/2$ (Lemma \ref{estim}b)) again shows that we can apply Theorem \ref{eks}a) with $(A',B,C)$.
\end{proof}

\begin{lemma}
Assume that $(A,B,C)$ is odd, at least one of $a,b$ is even, and $a,b,c$ are not all even. 
Then the statement of Theorem \ref{main} holds with elements $x,y,z$ all in $S_{c+1}$.
\end{lemma}
\begin{proof}
Assume for the sake of convenience that $a$ is even. This can be done without loss, as we will not use the assumption $a\le b$ here. 
Since $(A,B,C)$ can only be odd if exactly one or three out of $a,b,c$ are even, our assumptions already force $b$ and $c$ to be odd. Now replace $A,B$ and $C$ by the classes with the same number of non-trivial cycles in $S_{c+1}$.
In particular, the cycle structure of $C$ is $(c.1)$, and elements of $A$ have at least two fixed points on $\{1,...,c+1\}$ (since $a$ is even and $c$ is odd).
Since $(A,B,C)$ is odd, the strict inequality $ind(A)+ind(B)> c-1$ must hold (or otherwise $ind(A)+ind(B)+ind(C)=2(c-1)$),
so after replacing $A$ with the class $A'$ with one extra transposition, we have $ind(A')+ind(B)\ge c+1$ and $(A',B,C)$ is even. We can therefore apply Theorem \ref{eks}b) in $S_{c+1}$.
\end{proof}

Now the only cases that are left to treat are the case $a=b=c$ even, and the case that $(A,B,C)$ is odd and both $a$ and $b$ are odd. In the last case, obviously $c$ must be even.
We treat these cases in the following lemma, which therefore finishes the proof of Theorem \ref{main}. Its proof is slightly more involved than the previous ones, mainly since we now cannot apply Theorem \ref{eks} directly.
\begin{lemma}
\label{onlyceven}
Let $c$ be even, and either $a$ and $b$ odd or $a=b=c$.
Then the statement of Theorem \ref{main} holds with elements $x,y,z  \in S_{c+2}$ where $z$ is 
of cycle structure $(c.2)$. Furthermore, at least one of $x$ and $y$ has a fixed point on the support of the $c$-cycle of $z$.
\end{lemma}
\begin{proof}
The special case $a=b=c$ can be treated separately, via the permutations $$x=(1,...,c), y=(1,...,c-4,c-1,c-3,c+1,c+2), \text{ with}$$
$$xy = (\underbrace{1,3,...,c-5}_{c/2 - 2 \text{ odd integers}},c-1,c,\underbrace{2,4,...,c-4}_{c/2 - 2 \text{ even integers}}, c+1,c+2)(c-3,c-2).$$
In all other cases, $a$ and $b$ are odd. We will again split the proof into several cases, depending on how large $c$ is compared to $a$ and $b$. Our strategy is to first solve the problem for ``relatively small" $c$,
and then multiply such solutions with suitable cycles to obtain solutions for larger $c$. For this induction step, we will need the additional statement on fixed points in the lemma.
\begin{itemize}
 \item[Case 1:]
First, let $b> \frac{c+1}{2}$.\\
Once again, let the classes $A$ and $C$ be as defined at the beginning of Section \ref{proof}, now viewed as classes in $S_{c+1}$ (in particular, the elements of $A$ have at least one fixed point on $\{1,...,c+1\}$). 
Furthermore, let $B'$ be the class of $b-1$-cycles in $S_{c+1}$. Since $a$ is odd and $b-1$ and $c$ are even, the class triple $(A,B',C)$ is even. 

By Lemma \ref{estim}b), we always have $ind(A) \ge \frac{c}{2}$, with equality if and only if $a=\frac{c}{2}+1$ or $(a,c)=(3,8)$; 
and $ind(B')\ge \frac{c}{2}-1$, with equality if and only if $b=\frac{c}{2}+1$. Therefore, $ind(A)+ind(B')$ is at least $c-1$ and $\equiv c-1$ mod $2$ (as $(A,B',C)$ is even), which enforces $ind(A)+ind(B')\ge c+1$ unless
\begin{equation}
\label{ausnahmen} 
a=b=\frac{c}{2}+1  \text{ or } (a,b,c)=(3,5,8).
\end{equation}

So apart from these two exceptions, the conditions of Theorem \ref{eks}b) are satisfied. Therefore there are $x\in A$
and $y' \in B'$ such that $xy'\in C$, and $\langle x,y'\rangle$ is transitive on $\{1,...,c+1\}$. Say that the fixed point of $xy'$ is $c+1$.
Because of the transitivity, neither $x$ nor $y'$ fix $c+1$. But
then $y:=y' (c+1,c+2)$ is a $b$-cycle,\footnote{Here and below, we use the elementary fact that the product of a $k$-cycle and an $m$-cycle whose supports share exactly one point is a $(k+m-1)$-cycle.}
$xy = x y' (c+1, c+2) \in S_{c+2}$ has orbits $\{1,...,c\}$ and $\{c+1,c+2\}$, and $x$ (being in the class $A$, but not fixing $c+1$) fixes a point out of $\{1,...,c\}$.
 
The exceptional cases in (\ref{ausnahmen}) follow directly from 
$$x:=(1, 2, 3)(4, 5, 6)(7, 8, 9), y:=(1, 4, 8, 9, 10) \in S_{10}, \text{ with } xy=(1, 2, 3, 4, 5, 6, 8, 10)(7, 9),$$
and 
$$x:=(1,...,a)(a+1,...,2a), y:=(1,...,a-2,2a,2a-2) \in S_{c+2}=S_{2a}, \text{ (} a \text {  odd), with} $$
$$xy=(\underbrace{1,3,...,a-2}_{(a-1)/2 \text{ odd integers}},a-1,a,\underbrace{2,4,...,a-3}_{(a-3)/2 \text{ even integers}},2a,\underbrace{a+1,...,2a-3}_{a-3 \text{ integers}})(2a-1,2a-2),$$
of cycle structure $(2a-2.2)=(c.2)$.\\
This finishes Case 1.
\item[Case 2:]
Next, let $b\le \frac{c+1}{2}$ (i.e.\ $c\ge 2b-1$), but $c < 2b+a-2$. Then the triple $(a,b,c-(a-1))$ falls into the first case above, and therefore is realizable by some permutation triple 
$(x,y,z)$ in $S_{c-(a-1)+2}$.\\
Also, from the construction in Case 1, the element $x$ of order $a$ in this triple has at least one fixed point $d$ in the support of the $c-(a-1)$-cycle of $z$.
But then the $a$-cycle $\tau:=(d, c-a+4,...,c+2)$ has disjoint support with $x$ and shares exactly the point $d$ with the support of $z$, so $\tau x$ consists of exactly one 
$a$-cycle more than $x$, whereas $\tau xy$ is of cycle structure $(c.2)$. We therefore have realized the triple $(a,b,c)$, and the element $y$ has a fixed point on the support
of the $c$-cycle.
\item[Case 3:]
Finally, let $c \ge 2b+a-2$. Set $c':=c-(a+b-2)$. Note that the triple $(a,b,c')$ is still in ascending order and $c'$ is even. 
Assume by induction that the assertion has been proven for the triple $(a,b,c')$, with a permutation tuple $(x,y,z)$ realizing it in $S_{c-a-b+4}$. 
One of $x$ and $y$ - say, $x$, without loss - will have a fixed point $d$ 
on the support of the $(c-a-b+2)$-cycle of $z$. Set 
$$\rho:= (d, c-a+b+5,...,c-b+3) \text{ and } \tau:= (c-b+3,..., c+2).$$ Then $\rho$ is an $a$-cycle with support disjoint to the one of $x$, so $\rho x$ has just one $a$-cycle more than $x$. Similarly, $y\tau$
has just one $b$-cycle more than $y$. And $\rho(xy)\tau$ is of cycle structure $(c.2)$ (since $\rho$ shares exactly the point $d$ with the $c'$-cycle of $xy$, so $\rho xy$ is of cycle structure $(c'+a-1.2)$; 
and in the same way, $\tau$ shares exactly the point $c-b+3$ with the large orbit of $\rho xy$, so $\rho(xy)\tau$ is of cycle structure $(c'+a-1+b-1.2)=(c.2)$).\\
So we have realized the triple $(a,b,c)$, and $\rho x$ still has a fixed point on the support of the $c$-cycle.
\end{itemize}
\vspace{-7mm}
\end{proof}

\begin{remark}
 The above lemmata are sufficient to prove Theorem \ref{main}. However, stronger results seem to hold. Indeed, the statement of Theorem \ref{eks}b) seems to hold with the $(n-1)$-cycle
replaced by an element of cycle structure $(n-2.2)$. Proving this would considerably shorten the proof of Lemma \ref{onlyceven}.
\end{remark}
As an easy corollary from Theorem \ref{main}, one also gets the existence of tuples of arbitrary length $r\ge 3$ with prescribed orders and product $1$:
\begin{kor}
\label{arb_length}
Let $r\ge 3$ and $1<a_1,...,a_r\in \nn$. Set $n:=\max\{a_1,...,a_r\}+2$. Then there exist elements $x_1,...,x_r$ of $S_{n}$ such that $|x_i|=a_i$ for all $i\in \{1,...,r\}$ and $x_1\cdots x_r = 1$.
\end{kor}
\begin{proof}
By induction over $r$. Theorem \ref{main} shows the case $r=3$. If $r>3$, set $r':=\lfloor r/2\rfloor$,  and let $p$ be a prime with $n/2 < p \le n-2$ (such a $p$ exists for all $n\ge 7$, and for $n=5$).
Then the sets $\{0, 1,...,r'\}$, and $\{0,r'+1,r'+2,...,r\}$ are both of cardinality $<r$. Set $a_0:=p$, then $(a_0,a_1,...,a_{r'})$ and $(a_0, a_{r'+1},...,a_r)$ are both realizable in $S_n$ (with permutation tuples 
$(x_0,x_1,...,x_{r'})$ and $(y_0,x_{r'+1},...,x_r)$, each with product 1); but the only cycle type in $S_n$ that leads to an element of order $p$ is the $p$-cycle, so $x_1\cdots x_{r'}$ and $x_{r'+1}\cdots x_r$ are both $p$-cycles. 
By conjugating appropriately, we can actually assume that they are each other's inverse.

The only cases not covered are $n=4$ and $n=6$. If $n=4$, all $a_i$ must be equal to $2$, and the existence of arbitrarily long tuples of involutions with product one is clear (simply repeat tuples of lengths 2 and 3 
sufficiently often). For $n=6$, it can be checked directly that Theorem \ref{main} remains true if we only demand $1<a,b,c\le 5$, so the above induction argument works with $p=5$.
\end{proof}
\section{A topological interpretation}
\label{covers}
The above results translate immediately, via covering theory, to an existence result for branched coverings $f:R\to \mathbb{P}^1\cc$ of Riemann surfaces 
with prescribed ramification indices and bounded degree.

%\section{Introduction: Covers of the projective line}
To make this translation clear, we briefly recall the Hurwitz existence problem for coverings of the projective line $\mathbb{P}^1\cc$ and its connection to factorizations of permutations. 
Cf.\ \cite[Chapters 4 and 5]{Voe} for the basic theory.
\begin{defi}
Let $R$ be a compact connected Riemann surface. A non-constant holomorphic map $f:R\to \mathbb{P}^1\cc$ is called a {\it branched covering} (of the projective line).\\
There is an $n\in \nn$ such that, with the exception of finitely many points, $f$ is locally an $n$-to-$1$ map over a given point in $\mathbb{P}^1\cc$. This $n$ is called the {\it degree} of $f$.
The exceptional points are called the {\it branch points} of the covering.
\end{defi}
Covering space theory associates to every covering from a compact connected Riemann surface to $\mathbb{P}^1\cc$ an $r$-tuple of permutations 
$(\alpha_1,...,\alpha_r)$ such that $\alpha_1\cdots \alpha_r = 1$. The group generated by $\alpha_1,...,\alpha_r$ is a transitive permutation group of degree $n$, the degree of the covering,
called the {\it monodromy group} of the covering. Each $\alpha_i$ generates an {\it inertia group} at a branch point of the Galois closure of the covering.\\
Riemann's existence theorem assures that, for any permutation group $G$, the conditions $\alpha_1\cdots\alpha_r=1$ and $\langle\alpha_1,...,\alpha_r\rangle = G$ are also sufficient for the existence of a branched covering 
with monodromy group $G$.\\
The question of existence of coverings with a certain prescribed ramification behaviour is thereby reduced to the question whether, 
for given conjugacy classes $C_1,...,C_r$ of the symmetric group $S_n$ (in other words, for $r$ partitions of $n$, yielding the cycle structures of the respective classes), there are permutations 
$\alpha_i\in C_i$ with $\alpha_1\cdots \alpha_r = 1$ and $\langle \alpha_1,...,\alpha_r \rangle$ transitive. We call this the Hurwitz existence problem because Hurwitz 
first obtained the reduction to a problem about permutations in \cite{Hurw}.\\ 
An easy necessary condition for the existence of such permutations is given by the Riemann-Hurwitz genus formula:
\begin{satz}[Riemann-Hurwitz genus formula]
\label{riem-hurw}
Let $f:R\to \mathbb{P}^1\cc$ be a branched degree-$n$ covering of compact connected Riemann surfaces, with monodromy given by the permutation tuple $(\alpha_1,....,\alpha_r)$. 
Then the genus of $R$ is given by $g(R)=- (n-1) +\frac{1}{2}\sum_{i=1}^r ind(\alpha_i)$.\\
Here $g(R)$ is always a non-negative integer.
\end{satz}
There are several notable results on {\it sufficient} conditions for the existence of covers with given branch cycle structures, e.g.\ in \cite{Baranski}, \cite{EKS}, \cite{Thom}.

The following new existence result is a consequence of Theorem \ref{main}:
\begin{satz}
Let $r\ge 3$ and $1< a_1\le...\le a_r \in \nn$ be positive integers, and let $\{p_1,...,p_r\}\subset \cc\cup\{\infty\}$ be a set of cardinality $r$. Then there exists a branched covering $f:R\to \mathbb{P}^1\cc$ of compact
connected Riemann surfaces {\bf of degree at most $a_r+2$}, ramified exactly over $\{p_1,...,p_r\}$,
such that the inertia group at the point $p_i$ is generated by an element of order $a_i$ for all $i\in \{1,...,r\}$.\\
More precisely, $R$ and $f$ can be chosen such that all the preimages in $R$ of a given point $p_i$ have ramification index 1, 2 or $a_i$ ($i=1,...,r$).
\end{satz}
\begin{proof}
The existence follows readily from Corollary \ref{arb_length}. The refined statement on ramification indices is a direct consequence of the statement about cycle structures in Theorem \ref{main}
(and its application in the induction argument in the proof of Corollary \ref{arb_length}).
\end{proof}

\end{document}